\documentclass{amsart}
\usepackage{amsmath} 
\usepackage[utf8]{inputenc}
\usepackage{amssymb}
\usepackage{pgfplots}
\usepackage{tikz}
\usepackage{subcaption}
\usepackage{caption}
\usepackage{tikz-cd}
\usepackage{graphicx}
\usepackage{mathrsfs}
\usepackage{color}
\usepackage{xcolor}
\usepackage{accents}
\usepackage{amsthm}
\usepackage{scalerel}
\usepackage{float}
\usepackage{mathtools}
\usepackage{cite}
\usepackage{url}
\usepackage{array}
\usepackage{thmtools, thm-restate}
\usepackage{hyperref}

\newcommand\IZ{\mathbb{Z}}

\newcommand\IC{\mathbb{C}}
\newcommand\IN{\mathbb{N}}
\newcommand\IQ{\mathbb{Q}}
\newcommand\GL{\mathrm{GL}}
\newcommand\Aut{\mathrm{Aut}}
\newcommand\End{\mathrm{End}}
\newcommand\Out{\mathrm{Out}}
\newcommand\HE{\mathrm{HE}}

\newcommand\iso{\cong}
\newcommand\Mod{\mathrm{Mod}}

\newtheorem{theorem}{Theorem}[section]
\newtheorem{prop}[theorem]{Proposition}

\theoremstyle{definition}
\newtheorem{mydef}[theorem]{Definition}

\theoremstyle{definition}

\theoremstyle{definition}
\newtheorem{lemma}[theorem]{Lemma}

\theoremstyle{definition}

\theoremstyle{definition}
\newtheorem{corollary}[theorem]{Corollary}

\theoremstyle{definition}
\newtheorem{obs}{Observation}

\theoremstyle{definition}
\newtheorem{case}{Case}

\theoremstyle{definition}

\newtheorem{manualquestioninner}[theorem]{Question}
\newenvironment{manualquestion}[1]{%
  \renewcommand\themanualquestioninner{#1}%
  \manualquestioninner
}{\endmanualquestioninner}

\theoremstyle{definition}
\newtheorem{bsp}[theorem]{Example}

\AtBeginEnvironment{definition}{}

\makeatletter
\@namedef{subjclassname@2020}{\textup{2020} Mathematics Subject Classification}
\makeatother

\begin{document}

\title[Free Group Automorphisms via Virtual Homology Representations]{Detecting Free Group Automorphisms via Virtual Homology Representations}
\author{Emre Yüksel}
\address{University of Munich, Mathematisches Institut, Theresienstraße 39, 80333 Munich, Germany}
\email{yueksel.emre@gmail.com}

\subjclass[2020]{57M10, 57M60, 57K20}
	
\begin{abstract}
Let $F_n= F\langle x_1,...,x_n\rangle$ denote the free group of rank $n\ge 2$ and let $\mathrm{End}(F_n)$ be the endomorphism monoid of $F_n$. We show that automorphisms of $F_n$ are detected via the $\mathrm{End}(F_n)$-action on the first integral homology of finite characteristic covers of the wedge of $n\ge 2$ circles $R_n$. This gives a homological characterization of homotopy equivalences of $R_n$ that we utilize to show that $\mathrm{End}(F_n)$ is asymptotically linear. We extend these results by showing that the $\Out(F_n)$-action on the homology of iterated covers of a punctured surface $\Sigma_g^b$ of the same homotopy type as $R_n$ detects homeomorphisms of $\Sigma_g^b$ in homotopy classes of homotopy equivalences of $R_n$. \end{abstract} 

\maketitle

\tableofcontents

\section{Introduction}

Let $F_n$ denote the free group of rank $n$. It is a natural question in the theory of groups whether a given abstract group $G$ is linear, i.e. whether there exists a faithful representation of $G$ into the general linear group $\GL_n(\IC)$. For the automorphism group $\Aut(F_n)$ it was a long-standing open problem whether even $\Aut(F_2)$ is linear \cite{Dokovic}. It was eventually confirmed by Krammer \cite{Krammer} by showing that the braid group $B_4$ is linear as it was established in \cite{Formanek2} that $B_4$ admits a faithful representation into $\GL_n(\IC)$ if and only if $\Aut(F_2)$ admits a faithful representation into $\GL_{2n}(\IC)$. For $n\ge 3$ on the other hand, Formanek and Procesi \cite{Formanek} proved that $\Aut(F_n)$ is not linear.\\

For surfaces, it is still an open problem whether the mapping class group is linear. While this was answered positively by Bigelow and Budney \cite{Budney} and in the same year by Korkmaz \cite{Korkmaz} for the mapping class group $\Mod(\Sigma)$ of a surface $\Sigma$ of genus $g=2$, it is still an open question for surfaces of genus $g\ge 3$. A natural way to study $\Mod(\Sigma)$ is via its \emph{homology representation}, that is, the action of $\Sigma$-homeomorphisms on the first integral homology of $\Sigma$. The homology representation fits into a short exact sequence $$1\to \mathcal I_g\to \Mod(\Sigma)\to \operatorname{Sp}_{2g}(\IZ)\to 1$$ where $\mathcal I_g$ denotes the \emph{Torelli group}. While the homology representation is useful, much of the information about mapping classes is lost as the Torelli group $\mathcal I_g$ already exhibits the difficulties that appear when examining the mapping class group. Koberda \cite{Koberda} proved that much of this information can be recovered by examining the \emph{virtual homology representations} instead. That is, by studying the $\Mod(\Sigma)$-action on the first homology of finite regular covers of $\Sigma$, Koberda established that $\Mod(\Sigma)$ is asymptotically linear.\\

In this article we discuss the analogous situation for graphs and the automorphism group $\Aut(F_n)$. We identify $F_n$ with the fundamental group of the wedge of $n\ge 2$ circles $R_n$, i.e. the wedge of $n$ copies of $S^1$ based at the single basepoint $*$. We denote by $\mathrm{HE}(R_n,*)$ the group of basepoint preserving homotopy equivalences of $R_n$ and we identify the automorphism group $\mathrm{Aut}(F_n)$ with the group $\mathrm{HE}(R_n,*)$ modulo  homotopy, i.e $$\mathrm{Aut}(F_n)\iso \pi_0(\mathrm{HE}(R_n,*)).$$ If the basepoint is ignored, then the group $\pi_0(\HE(R_n))$ of homotopy classes of \emph{unbased} homotopy equivalences of $R_n$ is identified with the outer automorphism group $\Out(F_n)$ instead.\\

 For a topological space $X$, we denote its fundamental group often simply by a representative $\pi_1(X)$ of its isomorphism class with respect to the change of basepoint isomorphism. Whenever the choice of a basepoint $x_1\in X$ is required, we mention it explicitly and denote its fundamental group by $\pi_1(X,x_1)$. We say that $\pi_1(X)$ is a \emph{characteristic subgroup} of $F_n$, if it is invariant under the action of any $\psi\in \Aut(F_n)$. Whenever $\pi_1(X)\le F_n$ is characteristic, we say that $p\colon X\to R_n$ is a \textit{characteristic cover}. Observe that characteristic covers are regular and we therefore have a short exact sequence of groups $$1\to \pi_1(X)\to F_n\to \Gamma\to 1$$ where $\Gamma\iso F_n/\pi(X)$ is the associated deck group and we similarly refer to $\Gamma$ as a \emph{characteristic quotient}.  Any characteristic subgroup of $F_n$ with quotient $\Gamma$ naturally determines a well defined homomorphism $\Aut(F_n)\to \Aut(\Gamma)$. Each $\psi\in \Aut(F_n)$ therefore induces an automorphism of $\Gamma$ which will also be denoted by $\psi$. It will be clear from the context whether we consider an automorphism of $F_n$ or $\Gamma$.\\
 For any characteristic cover $X\to R_n$, the lifting criterion asserts that lifts of homotopy equivalences of $R_n$ to $X$ always exist but are well defined only up to deck transformations unless a basepoint in $X$ is chosen. Accordingly, we say that $p\colon X\to R_n$ is a \emph{based cover} whenever we fix a \emph{preferred basepoint} $x_1$ in the fiber of the basepoint $*$ of $R_n$.\\ 
 
We beginn by studying the \textit{virtual homology representations} of $\Aut(F_n)$. That is, for any finite characteristic (based) cover $(X,x_1)\to (R_n,*)$ we have a representation $$\Aut(F_n)\to \Aut(H_1(X,\IZ))= \GL(H_1(X,\IZ))$$

that is obtained by lifting an automorphism $\psi\in \Aut(F_n)$ to a representative $f\in \HE(R_n,*)$ of its associated homotopy class, which in turn lifts to an element $f'\in \HE(X,x_1)$, and then examining its action on the homology $H_1(X,\IZ)$. In this setting we prove the following graph-theoretic analog of \cite[Theorem 1.1]{Koberda}.

\begin{prop}\label{asymptotic_aut}Let $\psi\in \mathrm{Aut}(F_n)$. Let $\Gamma$ be a finite characteristic quotient of $F_n$ and let $X\to R_n$ be the associated based cover. If $1\not=\psi\in \mathrm{Aut}(\Gamma)$, then $\psi$ acts nontrivially on $H_1(X,\mathbb{Z})$. 
\end{prop}

We say that the subgroup $\pi_1(X)\le F_n$ (and thus the deck group $\Gamma$) is \emph{fully characteristic} if $\pi_1(X)$ also remains invariant under endomorphisms of $F_n$ and we denote by $\End(F_n)$ the monoid of $F_n$-endomorphisms.  By restricting the $\Aut(F_n)$-action on the homology of based covers to $\End(F_n)$,  we obtain a virtual homology representation of $\End(F_n)$ similar to $\Aut(F_n)$ by lifting an element $\phi\in \End(F_n)$ to a representative of the homotopy class of the basepoint preserving graph endomorphism $(X,x_1)\to (X,x_1)$ of a fully characteristic based cover $(X,x_1)\to (R_n,*)$ and study its action on the integral homology $H_1(X,\IZ)$.\\

A fundamental property of finitely generated free groups is that they are Hopfian, that is:  every epimorphism of $F_n$ is an automorphism. Our first main result shows that the virtual homology representations of $\End(F_n)$ detect automorphisms of $F_n$.

 \begin{theorem}\label{main_theorem}
	Let $\phi\in \mathrm{End}(F_n)$ and let $\mathcal K$ be the class of fully characteristic finite based covers of $R_n$. If $\phi$ acts by epimorphisms on $H_1(X,\IZ)$ for every  $X\in\mathcal K$, then $\phi$ is surjective.	
 \end{theorem}

If $\Gamma$ is a fully characteristic quotient of $F_n$, then analogous to the case of $\Aut(F_n)$, any $\phi\in \End(F_n)$ induces a well defined map $\phi\in \End(\Gamma)$. As a consequence of Theorem \ref{main_theorem}, we show that the virtual homology representations are asymptotically faithful.
\begin{corollary}\label{endo_faithful} Let $\phi\in \End(F_n)$ and let $X\to R_n$ be a finite fully characteristic based cover of graphs with deck group $\Gamma$. If $1\not=\phi\in \End(\Gamma)$, then $\phi$ acts nontrivially on $H_1(X,\IZ)$. \end{corollary}

Putting these results into perspective, we note that the residual finiteness of $F_n$ implies that $F_n$ contains many characteristic subgroups of finite index since each $1\not=w\in F_n$ projects nontrivially to some finite characteristic quotient. We will explicitly construct such an exhausting sequence of fully characteristic finite index subgroups of $F_n$ via iterated mod $q$ homology covers of $R_n$. In light of Theorem \ref{main_theorem}, we ask whether the $F_n$-automorphisms could in fact also be detected via the action of $\End(F_n)$ on a sequence of exhausting quotients that are fully characteristic.

\begin{manualquestioninner}\label{q1}
 Given a sequence $\{\Gamma_i\}_{i\in I}$ of fully characteristic quotients that exhaust $F_n$ and on which $\phi\in \mathrm{End}(F_n)$ acts by epimorphisms for every $\Gamma\in \{\Gamma_i\}$. Is this sufficient to conclude that $\phi$ is an epimorphism?
\end{manualquestioninner}

We give a negative answer to Question \ref{q1} by constructing non-surjective endomorphisms of $F_n$ that act by epimorphisms on an exhausting sequence of canonical nilpotent quotients that are obtained from the lower central series of $F_n$. \\

In the final sections we extend our considerations to surfaces. To this end, given a finite regular cover $\Sigma'\to\Sigma$ of surfaces with deck group $\Gamma$ and any loop $x\in \pi_1(\Sigma)$, we say that $x'\subset \Sigma'$ is a \emph{preferred elevation} of $x$, if $x'$ is the lift of a minimal power of $x$ to the cover $\Sigma'$ at the preferred basepoint in $\Sigma'$ and we refer to the images of the $\Gamma$-orbit of $x'$ just as \textit{elevations}. We give a new combinatorial proof of an Abelian cover version of a celebrated result of Scott \cite{Scott1},\cite{Scott2} that shows that closed loops on a topological oriented surface elevate to simple loops in (finite) towers of regular covers.\\

We denote by $\Sigma_g^b$ the compact oriented surface with positive genus $g$ and $b\ge 1$ punctures that has the same homotopy type as the graph $R_n$ for $n=2g+b-1$ and consider a class $\mathcal S$ of finite regular covers of $\Sigma_g^b$. Our final result relates the virtual homology representations of free group automorphisms and the mapping class group of $\Sigma_g^b$ by considering the virtual homology representations of $\Out(F_n)$. We show that if $\psi\in \Out(F_n)$ acts on all finite characteristic covers of $\Sigma_g^b$ by preserving the algebraic intersection pairing, then the homotopy equivalence of $R_n$ that induces $\psi$ is contained in the free homotopy class of a homeomorphism of the surface $\Sigma_g^b$.

 \begin{theorem}\label{mcg}
Let $\mathcal S$ be the class of finite characteristic covers of $\Sigma_g^b$. If $\psi\in \Out(F_n)$ acts on $H_1(\Sigma',\IZ)$ by preserving the algebraic intersection form in every cover $\Sigma'\in \mathcal S$, then the homotopy equivalence of $R_n$ that represents $\psi$ is homotopic to a homeomorphism of $\Sigma_g^b$.
 \end{theorem}

\section{Acknowledgements}

This article originated from my master's thesis at the University of Munich under the supervision of Sebastian Hensel. I am immensely grateful for his continuous support, tireless efforts and encouragement, as well as many helpful remarks.

\section{Asymptotic Linearity of $\Aut(F_n)$}

 Given a finite regular cover $X\to R_n$ with associated deck group $\Gamma\iso F_n/\pi_1(X)$, the conjugation action of $F_n$ on $\pi_1(X)$ induces an action on $\pi_1(X)^{ab}\iso H_1(X,\IZ)$ that descends to an action of $\Gamma$ and therefore endows $H_1(X,\IZ)$ with the structure of a $\mathbb{Z}[\Gamma]$-module. \\

We begin by examining the $\Gamma$-action on the rational homology $H_1(X,\IQ)$ of covers $X\to R_n$ and show that these actions are faithful and then extend our observation to integral homology. Our proof is a standard argument that is used in various proofs of the Chevalley-Weil formula \cite{CW} for graphs (sometimes referred to as Gaschütz Theorem). The proof that follows is essentially the same as in \cite{Koberda} and we state it here merely for the sake of clarity. Also compare \cite{GL},\cite{GLLM} for a detailed treatment of the present material. 
\begin{lemma}\label{Gaschuetz_Kor}
	If $X\to R_n$ is a finite regular cover with associated deck group $\Gamma$, then $\Gamma$ acts faithfully on $H_1(X,\mathbb{Z})$.
\end{lemma}

\begin{proof}
It is well known from covering space theory that if $1\not=g\in\Gamma$, then $g$ acts without fixed points on $X$ and its Lefschetz number $\tau(g)$ vanishes (cf. \cite[Theorem 2C.3]{Hatcher}). To be more specific, since $X$ is a graph, it has trivial homology in degree $n\ge 2$ and $1\not= g\in \Gamma$ acts trivially on $H_0(X,\IQ)$ since $X$ is connected. Let  $\chi(g)$ denote the character of the $\IQ[\Gamma]$-module $H_1(X,\IQ)$, then for the Lefschetz number of $g$ it follows \begin{align*}
		\tau(g) &= \mathrm{tr}\left(H_0(X,\IQ)\to H_0(X,\IQ)\right) - \mathrm{tr}\left(H_1(X,\IQ)\to H_1(X,\IQ)\right) \\ &= \mathrm{tr}(\mathrm{id}_{H_0(X,\IQ)}) - \chi(g) \\ &= 1 -\chi(g).
	\end{align*}
	Thus $1\not=g$ implies $\chi(g)=1$. Now let $g=1$. Then $$\chi(1) = \dim_{\IQ}H_1(X,\IQ)=\mathrm{rank}\ \pi_1(X)$$ as $\pi_1(X)$ is free. This shows that $\Gamma$ acts faithfully on $H_1(X,\IQ)$. In order to extend this observation to integral homology, note that the rational $\Gamma$-representation $H_1(X,\IQ)\iso H_1(X,\IZ)\otimes_{\IZ}\IQ$ naturally includes the integral $\Gamma$-representation $H_1(X,\IZ)$. More specifically, any $v\in H_1(X,\IQ)$ is a rational multiple of an integral vector $v_0\in H_1(X,\IZ)$. Accordingly, the faithful action of $\Gamma$ on $H_1(X,\IQ)$ restricts to a faithful action on $H_1(X,\IZ)$
	\end{proof}

The following result now relates the $\Aut(F_n)$-action on the deck group $\Gamma$ with the $\Aut(F_n)$-action on the first integral homology $H_1(X,\IZ)$.

\begin{lemma}\label{equi} Let $\psi\in \mathrm{Aut}(F_n)$ and let $X\to R_n$ be a finite regular cover with deck group $\Gamma$. Then $1=\psi\in \mathrm{Aut}(\Gamma)$ if and only if $\psi\in \mathrm{Aut}(F_n)$ acts $\Gamma$-equivariantly on $H_1(X,\IZ)$.
\end{lemma}

\begin{proof}
Assume that $\psi\in \mathrm{Aut}(F_n)$ acts trivially on  $\mathrm{Aut}(\Gamma)$, i.e. $\psi(w) = w\gamma_w$ for some $\gamma_w\in \pi_1(X)$. Since $F_n$ acts on $\pi_1(X)$ by conjugation, it follows that for $\omega\in F_n$ and $\gamma\in \pi_1(X)$ we have \begin{align}\psi(w\cdot \gamma) = \psi(w\gamma w^{-1}) = \psi(w)\psi(\gamma)\psi(w)^{-1} = w\gamma_w\psi(\gamma)\gamma_w^{-1}w^{-1}.\label{conj} \end{align}
The conjugation action restricts to $\pi_1(X)$ and is trivial in the free $\mathbb{Z}$-module $H_1(X,\IZ)\iso \pi_1(X)/[\pi_1(X),\pi_1(X)]$, hence $\gamma_w\psi(\gamma)\gamma_w^{-1}\ \equiv \psi(\gamma) \bmod [\pi_1(X),\pi_1(X)]$ which implies $w\gamma_w\psi(\gamma)\gamma_w^{-1}w^{-1}\equiv w\psi(\gamma)w^{-1} \bmod [\pi_1(X),\pi_1(X)].$
Thus by (\ref{conj}) we have $$\psi(w\gamma w^{-1}) \equiv  w\psi(\gamma)w^{-1}\bmod [\pi_1(X),\pi_1(X)],$$
which shows that $\psi$ acts $\Gamma$-equivariantly on $H_1(X,\IZ)$. Conversely, assume that the $\mathrm{Aut}(F_n)$ action on $H_1(X,\IZ)$ is $\Gamma$-equivariant, i.e. for all $w\in F_n$ we have $$\psi(w\cdot \gamma) \equiv w\cdot \psi(\gamma) \bmod [\pi_1(X),\pi_1(X)]$$ which is to say that $\psi(w)\psi(\gamma)\psi(w)^{-1}\equiv w\psi(\gamma)w^{-1}\bmod [\pi_1(X),\pi_1(X)].$ Thus $\psi(w)$ and $w$ act the same on $H_1(X,\IZ)$, that is, $\psi(w)w^{-1}$ acts like the identity on $H_1(X,\IZ)$. Lemma \ref{Gaschuetz_Kor} thus implies $\psi(w)w^{-1} \equiv 1 \bmod \pi_1(X)$. \end{proof}

Using Lemma \ref{equi} we are now ready to prove Proposition \ref{asymptotic_aut} from the introduction.	 
\begin{proof}[Proof of Proposition \ref{asymptotic_aut}]
	Assume $\psi$ acts trivially on $H_1(X,\mathbb{Z})$, i.e. $\psi(d) = d$ for any $d\in H_1(X,\mathbb{Z})$. For $\gamma\in \Gamma$ we thus have $\psi(\gamma\cdot d) =  \gamma\cdot\psi(d).$ The proposition now follows from Lemma \ref{equi}.
	 \end{proof}
	 
\section{Detecting Free Group Automorphisms}\label{detectepi}

In this section we shall prove Theorem \ref{main_theorem}. We want to extend our examination of the virtual homology representations of $\Aut(F_n)$ to $\mathrm{End}(F_n)$. To this end, we show that $F_n$-automorphisms are detected by the $\End(F_n)$-action on the homology of characteristic covers $X\to R_n$ in the sense that if an element $\phi\in \End(F_n)$ acts by inducing an epimorphism on the homology of all finite (fully characteristic) covers of $R_n$, then $\phi\in \Aut(F_n)$.  In view of the discussion that follows, recall that finitely generated subgroups of $F_n$ are realized by immersions of finite graphs into $R_n$. The reader who is unfamiliar with this correspondence should consult \cite{Stallings}.\\

Theorem \ref{main_theorem} can be proved in several ways. One could employ a Theorem of Marshall Hall \cite[Prop 3.10]{Lyndon} and show that the (finitely generated) image of a non-surjective $F_n$-endomorphism is a nontrivial proper free factor of $F_n$ and hence a direct summand in homology which then shows the claim. The following proof follows a similar approach. Recall that for a subgroup $G\le F_n$ that corresponds to a cover $X\to R_n$, we have a natural identification in homology $H_1(G,\IZ)\iso H_1(X,\IZ)$. Also note that if $e$ is an edge of a graph $X$, then the \textit{algebraic intersection} of homology classes with $e$ is a well defined linear functional $\widehat\iota(e,\cdot)\colon H_1(X,\IZ)\to\IZ$ that defines a cohomology class in $H^{1}(X,\IZ)$.

\begin{proof}[Proof of Theorem \ref{main_theorem}]
	Let $\phi\in \mathrm{End}(F_n)$ and assume that $\phi$ is not surjective, i.e. $\phi(F_n)=:E$ is a finitely generated proper subgroup of $F_n$ that is realized by an immersion of (finite) graphs $f\colon \mathcal G\to R_n$ satisfying $f_*\pi_1(\mathcal G)\iso E$. By \cite[Theorem 6.1]{Stallings} we can extend $\mathcal G$ to a finite cover $\widehat f\colon\widehat{\mathcal{G}}\to R_n$ such that $\widehat f_*\pi_1(\widehat{\mathcal{G}})\iso G$ by adding finitely many suitable edges without adding any new vertices, so that $E\subseteq G$.  \\
	
Consider the cellular chain complex of $\widehat{\mathcal{G}}$ and let $C_1^{\mathrm{cell}}(\widehat{\mathcal{G}})$ denote the free $\mathbb{Z}$-module that is generated by the $1$-cells of $\widehat{\mathcal{G}}$. Fix an edge $\widehat e\in C_1^{\operatorname{cell}}(\widehat{\mathcal{G}})$ that is missing from $\mathcal{G}$, i.e. $\widehat e$ is one of the edges that extends the immersion $\mathcal{G}\to R_n$ to the cover $\widehat{\mathcal{G}}\to R_n$. Now consider the cohomology class $[\rho]\in H^1(\widehat{\mathcal{G}})$ that is represented by the integer valued homomorphism $\rho\in \mathrm{Hom}(C_1^{\operatorname{cell}}(\widehat{\mathcal{G}}),\mathbb{Z})$ assigning each edge $e$ of $\widehat{\mathcal{G}}$ the algebraic intersection number $\widehat\iota(e,\widehat p)$ where $\widehat p$ denotes an interior point of the fixed edge $\widehat e$, i.e. we have a map
\begin{align*}\rho\colon C_1^{\operatorname{cell}}(\widehat{\mathcal{G}})\to \mathbb{Z},\ e\mapsto \widehat\iota(e,\widehat p)
\end{align*}

such that $$\rho(e)=\begin{cases} 0, & e\not=\widehat e \\ 1, & e = \widehat e.
	 	
	 \end{cases}$$
It follows that $\rho$ represents a nonvanishing cohomology class $[\rho]\in H^1(\widehat{\mathcal{G}},\mathbb{Z})$ that vanishes on $H_1(\mathcal{G},\mathbb{Z})$ and thus $H_1(\mathcal{G},\mathbb{Z})$ is a proper submodule of $H_1(\widehat{\mathcal{G}},\mathbb{Z})$. Since both $\mathcal{G}$ and $\widehat{\mathcal{G}}$ are Eilenberg-MacLane spaces,  it follows that $H_1(E,\mathbb{Z})$ is a proper submodule of $H_1(G,\mathbb{Z})$. Consequently, we have that $\phi_*H_1(G,\mathbb{Z})\subseteq H_1(E,\mathbb{Z})$ which implies that image of the inclusion induced homomorphism $(\phi|_G)_*\colon H_1(G,\mathbb{Z})\to H_1(G,\mathbb{Z})$ is contained in a proper free submodule of $H_1(G,\mathbb{Z})$.
	 \end{proof}
	 An immediate consequence of Theorem \ref{main_theorem} is given by the following easy observation.
	 \begin{corollary}
	Let $f\colon R_n\to R_n$ be some map and let $\mathcal K$ be the class of finite characteristic based covers of $R_n$. If $\phi\in \End(F_n)$ represents $f$ and acts on $H_1(X,\IZ)$ by epimorphisms for every $X\in \mathcal K$, then $f$ is a homotopy equivalence of $R_n$.
\end{corollary}
	 
We are now ready to prove Corollary \ref{endo_faithful} from the introduction that establishes the asymptotic linearity of $\End(F_n)$.
\begin{proof}[Proof of Corollary \ref{endo_faithful}]
	Assume $\phi\in \mathrm{End}(F_n)$ acts trivially on $H_1(X,\mathbb{Z})$ for all $X\in \mathcal K$. Proposition \ref{main_theorem} then implies that $\phi$ is surjective, hence $\phi$ is an $F_n$-automorphism that acts trivially (and thus $\Gamma$-equivariantly) on $H_1(X,\mathbb{Z})$ for every $X\in \mathcal K$. Lemma \ref{equi} then establishes that $\phi$ acts trivially on the Deck group $\Gamma$ for every cover $X\in \mathcal K$. The assertion now follows from Proposition \ref{asymptotic_aut}.\end{proof}
	
\begin{corollary}\label{corollary_endo} Let $\{X_i\}_{i\in I}$ be an exhausting sequence of finite fully characteristic based covers of $R_n$ and let $1\not=\phi\in \mathrm{End}(F_n)$. Then $\phi$ acts nontrivially on $H_1(X_i,\mathbb{Z})$ for some $i\in I$.
\end{corollary}

\begin{proof}
	Let $\phi\in \End(F_n)$. Since $\{X_i\}_{i\in I}$ is exhausting, we always find some quotient $\Gamma_i$ such that $\phi(\gamma)\not\equiv \gamma\bmod \pi_1(X_i)$. That is, $\phi$ acts nontrivially on $\Gamma_i$. By Theorem \ref{main_theorem} the claim follows.
\end{proof}

In the context of the various corollaries of Theorem \ref{main_theorem}, we want to construct an exhausting sequence of fully characteristic subgroups of $F_n$ explicitly. One way to do so is by taking iterated \textit{$\bmod\ q$} homology covers of $R_n$.

\begin{mydef}
The \emph{$\operatorname{mod}$ $q$ homology cover} of $R_n$ is the cover that is defined by the epimorphism $$\Psi\colon \pi_1(R_n)\to H_1(R_n,\mathbb{Z}/q\mathbb{Z})$$ that factors through the abelianization projection of $\pi_1(R_n)$ followed by the $\operatorname{mod}$ $q$ reduction homomorphism.\end{mydef}

From the definition it immediately follows that $\operatorname{mod}$ $q$ homology covers of $R_n$ are finite Abelian covers with deck group $\Gamma=H_1(R_n,\mathbb{Z}/q\mathbb{Z})\iso (\mathbb{Z}/q\mathbb{Z})^n$.

\begin{figure}[H]
\centering
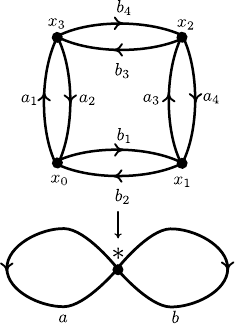
\caption{The mod $2$ homology cover $p\colon X\to R_2$.}
\label{mod_figure}
\end{figure}

\begin{bsp} Let $R_2$ be the wedge of $2$ circles $a,b$ with wedge point $*$ such that $\pi_1(R_2,*)$ is identified with $F_2$ (cf. Figure \ref{mod_figure}). The mod $2$ homology cover $p\colon (X,x_0)\to (R_2,*)$ is the cover defined by the epimorphism $F_2 \to \IZ/2\IZ\times\IZ/2\IZ$ and $X$ is therefore a $4$-fold Abelian cover with deck group the Klein four group $\mathbb{Z}/2\mathbb{Z}\times \mathbb{Z}/2\mathbb{Z}$.
\end{bsp}

\begin{lemma}\label{modq}
	The subgroup $G\le F_n$ that corresponds to a mod $q$ homology cover of $R_n$ is fully characteristic for any $q\in \IN, n\ge 2$.
\end{lemma}

\begin{proof}
	We verify that $\ker\Psi$ is invariant under any endomorphism of $F_n$. By definition we have $$\ker\Psi = \left\langle [x,y],z^q\mid x,y,z\in F_n\right\rangle.$$  Since commutators and powers of elements in $F_n$ are preserved by endomorphisms, the claim follows.\end{proof}

\begin{mydef}Let $$G_1:=\ker\left(\pi_1(R_n,*)\xrightarrow{\Psi} H_1(R_n,\mathbb{Z}/q\mathbb{Z})\right)$$ be the mod $q$ homology cover $X_1\to R_n$ such that $G_1\iso \pi_1(X_1,x_1)$ where $x_1$ denotes the preferred lift of the basepoint $*$. Similarly, let $$G_2:=\ker\Big(\pi_1(X_1,x_1)\to H_1(X_1,\mathbb{Z}/q\mathbb{Z})\Big)$$ be the mod $q$ homology cover of $X_2\to X_1$ with preferred basepoint $x_2\in X_2$. Setting $X_0:=R_n$ we therefore define inductively
 $$G_{k+1}:=\ker\Big(\pi_1(X_{k},x_{k})\to H_1\left(X_{k},\mathbb{Z}/q\mathbb{Z}\right)\Big),\ k= 0,1,2...$$ where $x_k$ is the preferred lift of $*$ to the iterated cover $X_{k}\to R_n$ and the characteristic subgroup $G_{k+1}$ is identified with $\pi_1(X_{k+1},x_{k+1})$ for all $k\in \IN$. This defines an infinite sequence of based covers $$\cdots\to  X_k \to X_{k-1}\to \cdots \to X_1\to R_n$$ that is called the \emph{tower of mod $q$ homology covers} of $R_n$.
\end{mydef}

We also recall the following standard result from combinatorial group theory.
\begin{lemma}[\!\!{\cite[Corollary 2.12]{Magnus}}]\label{exhausting_lemma} Let $$F_n\supset G_1\supset G_2\supset \cdots G_n\supset \cdots$$ be a descending sequence of distinct subgroups such that $G_{i+1}$ is characteristic in $G_i$, then the sequence $\{G_i\}_{i\in \IN}$ exhausts $F_n$.
\end{lemma}

\begin{proof} As subgroups of $F_n$, each $G_i$ is free. Recall that $\Aut(G_i)$ acts transitively on primitive elements of $G_i$. Since $G_{i+1}$ is characteristic in $G_i$, it is invariant under the action of any element in $\Aut(G_i)$. Thus if $G_{i+1}$ were to contain a primitive element of $G_i$, then $G_i=G_{i+1}$ which contradicts the assumption. The claim is now an application of \cite[Theorem 2.12]{Magnus} that asserts that if no $G_{i+1}$ contains any primitive element of $G_i$, the descending sequence $\{G_i\}_{i\in \IN}$ exhausts $F_n$.
\end{proof}

Note that all mod $q$ homology covers correspond to finitely generated free groups as they are finite covers of $R_n$. Lemma \ref{modq} establishes that the subgroup $G_{i+1}$ that corresponds to the mod $q$ homology cover $X_{i+1}\to X_i$ is fully characteristic in $G_i$ and of finite index since each $X_i$ is a finite cover of the finite graph $R_n$. Applying Corollary \ref{corollary_endo} we therefore obtain:

\begin{corollary}
	Let $\{X_i\}_{i\in \IN}$ be the tower of mod $q$ homology covers of $R_n$. Then $1\not=\phi\in \mathrm{End}(F_n)$ acts nontrivially on $H_1(X_i,\mathbb{Z})$ for some $i\in \IN$.
\end{corollary}

\section{The $\mathrm{End}(F_n)$-action on the Lower Central Series}\label{sec4}

In view of the implication of Proposition \ref{main_theorem} we now set out to draw a comparison between the action of $\mathrm{End}(F_n)$ on fully characteristic quotients of $F_n$ and the virtual homology representations. Recall that we raised the following question in the introduction:

\begin{manualquestion}{\ref{q1}}
 Given a sequence $\{\Gamma_i\}_{i\in I}$ of fully characteristic quotients that exhaust $F_n$ and on which $\phi\in \mathrm{End}(F_n)$ acts by epimorphisms for every $\Gamma\in \{\Gamma_i\}$. Is this sufficient to conclude that $\phi$ is an epimorphism?
\end{manualquestion}

We tackle Question \ref{q1} by studying the action of $\mathrm{End}(F_n)$ on the quotients that are derived from the lower central series of $F_n$. In particular, we demonstrate that there exist endomorphisms of $F_n$ that fail to be epimorphisms of $F_n$ despite inducing epimorphisms on all canonical nilpotent quotients $F_n/F_n^{(k)}$ were $F_n^{(k)}$ denotes the $k$th term of the lower central series. In this context, we reflect on a classical result of Magnus \cite{Magnus35} that establishes that the lower central series of $F_n$ is residually nilpotent and thus produces an exhausting sequence of fully characteristic quotients.\\

Observe that for any group $G$ with fully characteristic subgroup $K\le G$ and quotient $\Gamma=G/K$, the restriction of $\phi$ to $K$ induces an endomorphism $\phi|_K\in \mathrm{End}(K)$ which naturally induces a well defined map $\mathrm{End}(G)\to \mathrm{End}(\Gamma)$, that assigns each $\phi\in \mathrm{End}(G)$ to an endomorphism $\overline\phi\in \mathrm{End}(\Gamma)$. The following statement is trivial but we include it here for the sake of clarity in the arguments that follow.

\begin{lemma}\label{2out3}
	If $\phi|_K$ and $\overline\phi$ are both surjective, then $\phi$ is surjective.
\end{lemma}

\begin{lemma}\label{id} Let $F_n$ be the free group of finite rank $n$ with free basis $X=\{x_1,...,x_n\}$ and let $\phi\in \mathrm{End}(F_n)$ be an endomorphism that acts on $X$ by conjugating each basis element $x_i$ with some word in $X$, that is $$\phi(x_i)=\alpha_i^{-1}x_i\alpha_i,\ \alpha_i\in F_n.$$ Then $\phi$ acts trivially on all nilpotent quotients $F_n^{(k)}/F_n^{(k+1)}$ for $k\in \IN$, where $F_n^{(k)}$ denotes the $k$th term of the lower central series of $F_n$.
\end{lemma}

For the proof of Lemma \ref{id} we require the following commutator identities for elements in the nilpotent quotients.

\begin{lemma}\label{identities2} Let $x,y\in F_n, f\in F_n^{(k)}$, $k\in \IN$. Then it holds:
	\begin{itemize}\setlength{\itemindent}{0.5em}
	\item[(\ref{identities2}.1)] $[f,xy] \equiv [f,y][f,x] \bmod F_n^{(k+2)}$ 
	\item[(\ref{identities2}.2)] $[f, yxy^{-1}]\equiv [f,x] \bmod F_n^{(k+2)}$
	\end{itemize}
\end{lemma}

\begin{proof} By the Witt-Hall identities \cite[Ch. 5.2]{Magnus} we have $$[f,xy] = [f,y][f,x][f,x,y] \equiv [f,y][f,x] \bmod F_n^{(k+2)}$$ since $[f,x,y]=[[f,x],y] \equiv 1 \bmod F_n^{(k+2)}$. Thus using (\ref{identities2}.1) provides \begin{align*}
		[f,y^{-1}xy]\equiv [f, xy][f,y^{-1}] &\equiv [f,y][f,x][f,y^{-1}] \bmod F_n^{(k+2)}.	\end{align*} 	
	Moreover, since $F_n^{(k+1)}/F_n^{(k+2)}$ is Abelian, it follows that
	\begin{align*}
		[f,y][f,x][f,y^{-1}] \equiv [f,y][f,y^{-1}][f,x] \stackrel{(\ref{identities2}.1)}{\equiv}[f, y^{-1}y][f,x]\equiv [f,x]\bmod F_n^{(k+2)}.
	\end{align*}
\end{proof}

\begin{proof}[Proof of Lemma \ref{id}]
	We give a proof by induction. Observe that we can write the image $\phi(x_i)$ of any basis element $x_i$ as a product of the form $x_iC_i$ for some commutator $C_i\in F_n^{(2)}$ since $$\phi(x_i) =\alpha_i^{-1}x_i\alpha_i = x_ix_i^{-1}\alpha_i^{-1}x_i\alpha_i = x_i[x_i,\alpha].$$ It follows that $\phi(x_i)\equiv x_i \bmod F_n^{(2)}$ and thus $\phi=1 \in \mathrm{End}(F_n/F_n^{(2)})$. Therefore assume the assertion is true for $k\in \IN$ such that $$\phi[\omega^{(k-1)},x_i] \equiv [\omega^{(k-1)},x_i] \bmod F_n^{(k+1)}$$ given $\omega^{(k-1)}\in F_n^{(k-1)}$, which is equivalent to   $\phi = 1\in \mathrm{End}(F_n^{(k)}/F_n^{(k+1)})$. We thus have to verify that for $\omega^{(k)}\in F_n^{(k)}$ it holds that $$\phi[\omega^{(k)},x_i]\equiv [\omega^{(k)},x_i]\bmod F_n^{(k+2)}.$$ Observe that $\phi[\omega^{(k)},x_i]=[\omega^{(k)} f^{(k+1)},\alpha_i^{-1}x_i\alpha_i]$ for $f^{(k+1)} \in F_n^{(k+1)}$. We therefore have:
	\begin{align*} [\omega^{(k)},\alpha^{-1}x_i\alpha_i][\omega^{(k)},\alpha^{-1}x_i\alpha_i,f^{(k+1)}]&[f^{(k+1)},\alpha^{-1}x_i\alpha_i] \equiv \\ \equiv &[\omega^{(k)},\alpha^{-1}x_i\alpha_i][f^{(k+1)},\alpha^{-1}x_i\alpha_i] \\ \equiv &[\omega^{(k)},\alpha^{-1}x_i\alpha_i] \bmod F_n^{(k+2)}.\end{align*}
	Finally, using Lemma \ref{identities2} (\ref{identities2}.2) we obtain \begin{align*}
		[\omega^{(k)},\alpha^{-1}x_i\alpha_i]\equiv [\omega^{(k)},x_i]\bmod F_n^{(k+2)} 
	\end{align*}
	which shows that $\phi[\omega^{(k)},x_i]\equiv [\omega^{(k)},x_i]\bmod F_n^{(k+2)}$ as desired. The claim follows.
	
\end{proof}

By using Lemma \ref{id} we observe the following:

\begin{lemma}\label{epi_miss}
	Let $\phi\in \mathrm{End}(F_n)$ be the endomorphism $\phi(x_i) = \alpha_i^{-1}x_i\alpha_i$, $\alpha_i\in F_n$. Then the induced endomorphism $\phi\in \mathrm{End}(F_n/F_n^{(k)})$ is surjective for all $k\in \IN$. \end{lemma}

\begin{proof}
	First observe that the quotient $F_n/F_n^{(k)}$ fits into a short exact sequence for all $k\in \IN$:
$$1\to F_n^{(k)}/F_n^{(k+1)}\to F_n/F_n^{(k+1)}\to F_n/F_n^{(k)}\to 1.$$ 

We give a proof by induction on $k$. For $k=1$ there is nothing to show. Thus consider the short exact sequence 

$$1\to F_n^{(2)}/F_n^{(3)}\to F_n/F_n^{(3)}\to F_n/F_n^{(2)}\to 1.$$ 

By Lemma \ref{id} we know that $\phi$ acts trivially on both outermost groups $F_n/F_n^{(2)}$ and $F_n^{(2)}/F_n^{(3)}$. Lemma \ref{2out3} thus implies that $\phi \in \mathrm{End}(F_n/F_n^{(3)})$ is surjective. For the induction step assume that $\phi\in \mathrm{End}(F_n/F_n^{(k)})$ is surjective for $k\in \mathbb N$ and consider the short exact sequence 

$$1\to F_n^{(k)}/F_n^{(k+1)}\to F_n/F_n^{(k+1)}\to F_n/F_n^{(k)}\to 1.$$ 

Again, by Lemma \ref{id} we know that $\phi$ is the identity on $F_n^{(k)}/F_n^{(k+1)}$. The exact same argument as before establishes that  $\phi\in \mathrm{End}(F_n/F_n^{(k+1)})$ is surjective, which shows the claim.
\end{proof}	

The following result now provides a negative answer to Question \ref{q1}. Going forward, recall that for any choice of a basis $X=\{x_1,...,x_n\}$ for $F_n$, we can assign to any cyclically reduced word $\omega\in F_n$ of length $l$ an associated Whitehead graph $\mathrm{Wh}(\omega,X)$ whose $2l$ vertices are labelled by $x_i,x_i^{-1}$ for each occurring letter $x_i$ in $\omega$, so that for any two successive letters $x_ix_j$ there is an edge joining $x_i$ to $x_j^{-1}$. The reader who is not familiar with the notion of Whitehead graphs should consult \cite{Whitehead_Stallings}. 

\begin{prop}
	There exists a nontrivial cyclically reduced word $\alpha_i\in F_n$ such that the endomorphism $\phi\colon F_n\to F_n,\ \phi(x_i)= \alpha_i^{-1}x_i\alpha_i$ that conjugates some basis element $x_i$ by $\alpha_i$ is not an epimorphism of $F_n$.\end{prop}

\begin{proof}
	Let $X=\{x_1,...,x_n\}$ be the fixed basis of $F_n$. Consider the Nielsen automorphism that acts on the primitive element $x_1$ via $$x_1\mapsto x_1x_2,\ x_i\mapsto x_i,\ i\not=1$$ such that it carries $X$ to the basis $X'=\{x_1x_2,x_2,...,x_n\}$. Let $\phi\in \mathrm{End}(F_n)$ be an endomorphism that conjugates  $x_1$ by a nontrivial $\alpha_1\in F_n$ while leaving all other basis elements invariant, i.e. $$x_1\mapsto \alpha_1^{-1}x_1\alpha_1,\ x_i\mapsto x_i,\ i\not= 1,\ 1\not=\alpha_1\in F_n.$$ It follows that $\phi$ acts on $X'$ via $$\phi(x_1x_2) = \alpha_1^{-1} x_1\alpha_1x_2,\ \phi(x_i) = x_i,\ i=2,...,n.$$ 
	
	Choose $\alpha_1 =x_{i_1}x_{i_2}\cdots x_{i_k}$ to be a cyclically reduced word such that its Whitehead graph $\mathrm{Wh}(\alpha_1,X)$ is connected and admits no cut vertex and subject to the condition that $\alpha_1$ satisfies $x_{i_1}\not= x_1^{\pm 1},\ x_{i_k}=x_1^{-1}$. For instance, for $n\ge 2$ we may choose (cf. Figure \ref{no_cut}) the word $$\alpha_1 = x_nx_{n-1}\cdots x_2x_1x_n^{-2}x_{n-1}^2\cdots x_2^{-1}x_1^{-1}.$$

\begin{figure}[H]
\centering
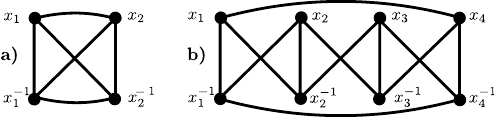
\caption{The Whitehead graph $\mathrm{Wh}(\alpha_1,X)$ in Figure a) for the cylically reduced word $\alpha_1 = x_2x_1x_2^{-2}x_1^{-2}$ in the rank two free group $F\langle x_1,x_2\rangle$ and in Figure b) for $\alpha_1 = x_4x_3x_2x_1x_4^{-2}x_3^{-2}x_2^{-2}x_1^{-2}$ in $F\langle x_1,x_2,x_3,x_4\rangle$. Both graphs are connected and admit no cut vertices.}
\label{no_cut}
\end{figure}

As a consequence, $\phi(x_1x_2)$ is cyclically reduced and the Whitehead graph $$\mathrm{Wh}(\alpha_1^{-1}, X)=\mathrm{Wh}(\alpha_1,X)$$ is a subgraph of $\mathrm{Wh}(\phi(x_1x_2), X)$. Indeed, since $\alpha_1$ is cyclically reduced, the relation $x_{i_k}=x_1^{-1}$ establishes that $\phi(x_1x_2)$ is cyclically reduced and that $\mathrm{Wh}(\phi(x_1x_2),X)$ contains a copy of $\mathrm{Wh}(\alpha_1,X)$ since the subsequence $x_{i_1}^{-1}x_1$ in the image $$\phi(x_1x_2)=\left(x_{i_k}^{-1}x_{i_{k-1}}^{-1}\cdots x_{i_2}^{-1}x_{i_1}^{-1}\right)x_1\Big(x_{i_1}x_{i_2}\cdots x_{i_k}\Big)x_2$$ generates the final edge of $\mathrm{Wh}(\alpha_1^{-1},X)$ inside $\mathrm{Wh}(\phi(x_1x_2),X)$ that might otherwise be missed upon concatenation in $\phi(x_1x_2)$. Furthermore, both conditions $x_{i_1}\not= x_1^{\pm 1},\ x_{i_k}=x_1^{-1}$ ensure that no undesired cancellation occurs in $\phi(x_1x_2)$. The absence of cut vertices and the connectedness from $\mathrm{Wh}(\alpha_1,X)$ is inherited by $\mathrm{Wh}(\phi(x_1x_2))$. By \cite[Corollary 2.5]{Whitehead_Stallings} it follows that $\phi(x_1x_2)$ is not separable which implies that $\phi$ does not preserve primitivity. Since $\Aut(F_n)$ acts transitively on primitive elements and
 $F_n$ is Hopfian, it follows that $\phi$ is not surjective. \end{proof}

\section{Lifting Closed Curves on Surfaces to Towers of Covers}\label{sec5}

We now consider a compact orientable surface $\Sigma$ of positive genus $g$ (possibly with boundary) with a fixed basepoint $s\in \Sigma$ in the interior and negative Euler characteristic $\chi(\Sigma)$. For the purpose of our results we do not distinguish between boundary components and punctures. In the punctured case we usually denote the genus $g$ surface with $b$ boundary components by $\Sigma_g^b$ and identify its fundamental group $\pi_1(\Sigma_g^b,s)$ with the free group $F_{2g+b-1}$ by fixing a basis $\{x_1,...,x_{2g+b-1}\}$ for $n=2g+b-1$. For any cover $\Sigma'$ of $\Sigma$ we denote by $\overline{\Sigma'}$ the closed surface that is obtained from $\Sigma'$ by filling in the punctures.\\

 To fix some notation, a \textit{closed curve} $\gamma$ on a surface $\Sigma$ will be any continuous map $\gamma\colon S^1\to \Sigma$ and we always understand $\gamma$ to be identified with its image. We say that the closed curve $\gamma$ is \textit{simple} if it's an embedding and if it is not based on a chosen basepoint in $\Sigma$, otherwise we refer to the (simple) closed curve $\gamma$ as (simple) \textit{loop} instead.  We say that the closed curve $\gamma\in \pi_1(\Sigma)$ is \textit{nonperipheral} if it is not freely homotopic to a closed curve in the neighbourhood of a boundary component. An \textit{arc} $\alpha\subset \Sigma$ is a continuous map $\alpha\colon [0,1]\to \Sigma$. We allow the arc $\alpha$ to intersect itself and accordingly say that $\alpha$ is \textit{simple} if it's an embedding.  \\

 Given any two closed curves $x,y$ on $\Sigma$, we will denote by $i(x,y)$ the \emph{geometric intersection number} of $x,y$, i.e. the minimal number of intersections of representatives of the free homotopy classes of $x$ and $y$ respectively. The \emph{algebraic intersection number} $\widehat\iota(\cdot,\cdot)$ is the alternating non-degenerate bilinear form on $H_1(\Sigma,\IZ)$ that is defined on homology classes of two transverse oriented simple closed curves on $\Sigma$ by counting the intersection indices according to the orientation of the respective curves. We will always assume any closed curve to be nonhomotopic to powers of peripheral simple closed curves unless otherwise stated. \\

The following result is the Abelian cover version of a well known theorem of \cite{Scott1}, \cite{Scott2}. Our proof is based on a proof-sketch in an unpublished preprint of Hensel \cite{Hensel}. We have contributed by extending the arguments and working out details where we considered necessary. A version for $p$-group covers of the following result can be found in \cite[Theorem 3.3]{Boggi}.

\begin{lemma}\label{scott} Let $\Sigma$ be a closed surface of positive genus $g$ and let $x\in \pi_1(\Sigma,s)$ be any loop. Then there exists some iterated finite Abelian cover $\Sigma'$ so that the elevation of $x$ to $\Sigma'$ does not have transverse self-intersections.
\end{lemma}

The proof of Lemma \ref{scott} uses various observations from covering space theory which are discussed below. We will show the existence of the desired tower of covers inductively while successively decreasing the self-intersection number of lifts of $x$. 

\begin{obs} Let $x\subset \Sigma$ be a loop that is not simple and let $y$ be a simple  subloop of $x$. If $\Sigma'\to \Sigma$ is a cover to which $x$ lifts but $y$ does not, then a lift $x'\subset \Sigma'$ of $x$ has a smaller self-intersection number than $x$. Compare e.g. \cite[Lemma 2.1]{Malestein} for a detailed treatment of this observation. \end{obs}

\begin{obs}\label{obs2} Let $x\subset \Sigma$ be a separating simple closed curve. Then there exists an Abelian cover $\Sigma'\to \Sigma$, so that elevations of $x$ to $\Sigma'$ are not separating. There are various ways to see this. Recall that the separating simple closed curves are contained in finitely many $\Mod(\Sigma)$-orbits, parametrised by the topological type of the splitting of the surface $\Sigma$ they induce (cf. \cite[1.3]{Primer}). Therefore assume that $x$ is a separating curve of some fixed genus $g'$. Consider a bounding pair of curves $y,z$ of the same genus $g'$ as $x$. Let $\{\alpha_1,\beta_1,...,\alpha_g,\beta_g\}$ be the standard configuration of curves for $H_1(\Sigma_g,\IZ)$. We can modify $y,z$ to a new curve $w$ of genus $g'$ by band-summing so that $w$ has geometric intersection $i(w,\alpha_g)=2$ and algebraic intersection $\widehat\iota([w],[\alpha_g])=0$. Band-summing has no effect in homology and thus $[w]=[y]+[z] = 0$, showing that $w$ is a separating curve of genus $g'$ and therefore an element of the same $\Mod(\Sigma)$-orbit of $x$. Now consider the order $2$ cyclic cover that is obtained by counting algebraic intersection mod $2$ with $\alpha_g$,  defined by the map $$\pi_1(\Sigma)\to H_1(\Sigma,\mathbb{Z})\to \IZ_2,\ x\mapsto [x]\mapsto\widehat\iota([x],[\alpha_g])\bmod 2.$$Then $w$ elevates to $2$ copies of nonseparating simple closed curves (cf. Figure \ref{lift_figure}). 
\end{obs}

\begin{figure}[H]
\centering
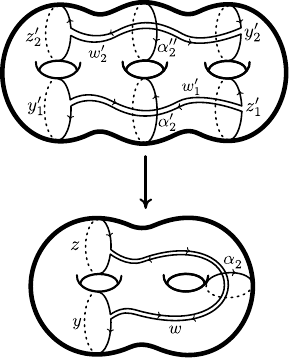
\caption{Band-summing a bounding pair $y,z$ of genus $g'=0$ to a separating curve $w$ of the same genus that lifts to nonseparating curves in the order $2$ cyclic cover. The homologous curves $y,z$ lift to nonhomologous curves $y_1',z_1'$.}
\label{lift_figure}
\end{figure}

\begin{obs}
	Suppose that $y,z$ are two disjoint simple closed curves so that $[z]=0$ or $[z]=[y]$. Then there exists an Abelian cover to which $y$ and $z$ lift but their lifts are not homologous and not nullhomologous. Namely, if $y,z$ are (disjoint) homologous nonseparating curves, then $y,z$ define a bounding pair. By Observation 2 it follows that there exists a finite cyclic cover $\Sigma'$ such that lifts $y',z'$ to $\Sigma'$ are no longer homologous (cf. Figure \ref{lift_figure}). If $z$ is separating, then $z$ lifts to a nonseparating curve in a finite cyclic cover by Observation $2$.  \end{obs}
	
We are now set to prove Lemma \ref{scott}.

\begin{proof}[Proof of Lemma \ref{scott}] Our goal is to show the existence of a finite regular cover $\Sigma'$ so that $x\subset \Sigma$ lifts to a simple closed curve by iterating through a tower of Abelian covers, while successively decreasing the number of self-intersections of lifts of $x$. In order to do so, we consider the following cases.

\begin{case}\label{case1} $[x]=0$. Suppose that there is any simple subloop $y\subset x$ such that $[y]\not= 0$. Then there exists a finite Abelian cover to which $y$ does not lift but $x$ does since $[x]=0$ is contained in the kernel of the defining epimorphism $\pi_1(\Sigma)\to \Gamma$ for any finite Abelian quotient $\Gamma$ that factors through the abelianization $H_1(\Sigma,\IZ)$. If there is no simple subloop $y$ so that $[y]\not=0$, we can pass to a cyclic cover $\Sigma'$ as in Observation \ref{obs2} to which $y$ lifts to a nonseparating simple subloop $y'$ of $x'$. If the lift of $x$ is nullhomologous in $\Sigma'$, the previous argument applies. Otherwise consider the following case.
\end{case}

\begin{case}\label{case2} $[x]\not=0$. A classical result from surface topology (see e.g. \cite[Prop. 6.2]{Primer}) states that every nontrivial homology class in $H_1(\Sigma,\mathbb{Z})$ is the multiple of some primitive class that is represented by a simple closed curve. Thus there exists a collection of simple closed curves $\{\alpha_i,\beta_i\}_i$ intersecting in the standard pattern so that we may assume without loss of generality that $x$ defines a multiple of $[\alpha_1]$ (cf. Figure \ref{primitive} for an example). In particular, $x$ has algebraic intersection nonzero only with $\beta_1$ since $\widehat\iota([\alpha_1],[\beta_1]) = 1$ and $\widehat\iota([\alpha_1],[\beta_j]) = 0$ for all $j\not= 1$. Thus if $[x]=m[\alpha_1]$ for some nonzero $m\in \mathbb{Z}$, we have
	$$\widehat\iota([x],[\beta_j]) = \widehat\iota(m[\alpha_1],[\beta_j]) = m\cdot \widehat\iota([\alpha_1],[\beta_j]) = \begin{cases}
		m, &\text{if} j = 1 \\ 0, &\text{if}\ j\not= 1.
	\end{cases}$$

	Now take a simple subloop $y\subset x$ with $[y]\not= 0$. Such a loop exists because otherwise $x$ would decompose into a collection of nullhomologous subcurves which contradicts $[x]\not= 0$.

\begin{figure}[H]
\centering
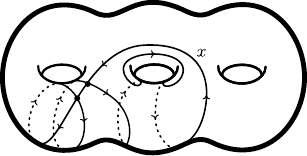
\caption{The closed curve $x=\alpha_1^2\beta_2\alpha_2\beta_2^{-1}\alpha_2^{-1}$ defines a multiple of $[\alpha_1]\in H_1(\Sigma,\IZ)$.}\label{primitive}
\end{figure}

If $y$ has nonzero algebraic intersection with any $\alpha_i$ or $\beta_j$, $j\not= 1$, then there is an Abelian cover to which $x$ lifts and $y$ does not and we are done. Indeed, assume without loss of generality that $\widehat\iota([y],[\alpha_2])=k\not= 0$. Then we may take the cyclic cover of order $l$ for $l,k$ coprime by counting modular algebraic intersection with $\alpha_2$. Since $\widehat\iota([x],[\alpha_2]) \bmod l =0$ but $\widehat\iota([y],[\alpha_2]) \bmod l \not= 0$ it follows that $x$ lifts but $y$ does not. Therefore assume that \begin{align}\widehat\iota([y],[\alpha_i])=\widehat\iota([y],[\beta_j])=0,\quad \forall \alpha_i,\beta_j,\ j\not=1.\end{align} Since $y$ is simple and $[y]\not=0$, equation (2) implies that $[y]=\pm[\alpha_1]$. Now if $[x]$ is a proper multiple of $[y]$, there exists a finite cyclic cover by counting algebraic intersection with $\beta_1$ to which $x$ lifts but $y$ does not and we are done. If $[x]$ is not a proper multiple of $[y]$, we continue by considering the following (sub)cases.\end{case} 
	 
	\textbf{Case 2A.} Consider a minimal subloop $z\subset x$ which intersects $y$ only in its endpoints and let $z'\subset z$ be a simple subloop. Assume $[z']\not= 0$ and $[z']\not=\pm[y]$. By acting with a suitable mapping class we may assume without loss of generality that $[z']=[\beta_j]$ for some $j\not=1$, say $j=2$. Then there exists a cyclic cover, obtained by counting algebraic intersection with $\alpha_2$ to which $x$ lifts but $z'$ does not, and we are done. Otherwise, if $[z']= 0$ or $[z']=\pm[y]$ we can apply Observation $3$ to find some Abelian cover $\Sigma'$ to which $z$ and $y$ lift but in this cover this situation will no longer occur and we exhausted all other possibilities. It now remains to consider the case in which $z$ and $y$ admit additional points of intersection.\\
	
	\textbf{Case 2B.} Assume the subloop $z$ intersects $y$ not just in its endpoints. If there is a simple subloop $z'\subset z$, then by Case 2A we are done. Otherwise there exists a minimal simple subarc $z'\subset z$ that intersects $y$ only in its (distinct) endpoints (cf. Figure \ref{case2b}) and returns to the same side of $y$. The endpoints of $z'$ divide $y$ into two simple subarcs $y',y''\subset y$. We may assume that $y'$ is the subarc that matches the orientation of $z'$ at the points of intersection, otherwise we pass to $y''$. Then $z'\cup y'$ defines an oriented simple subloop of $x$ and by the preceding arguments we find a finite cover to which $x$ lifts but $z'\cup y'$ does not, thus reducing the number of intersections of elevations of $x$. This finishes the proof.
	\end{proof}

\begin{figure}[H]
\centering
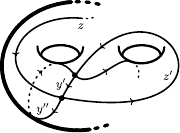
\caption{$z'\cup y'$ defines a simple subloop of $x$.}\label{case2b}
\end{figure}

	The following corollary now extends the results of Lemma \ref{scott} to surfaces with boundary.
	
	\begin{corollary}
		Let $\Sigma_g^b$ be a surface of positive genus $g$ and $b\ge 1$ boundary components and let $x\subset \Sigma_g^b$ be a nonperipheral closed curve. Then there exists some iterated finite regular cover of $\Sigma_g^b$ so that $x$ elevates to a simple closed curve.
	\end{corollary}
	
	\begin{proof}
	Let $D(\Sigma_g^b)$ denote the double of the surface $\Sigma_g^b$, i.e. the connected sum $\Sigma_g^b\# {\Sigma_g^b}'$ of $\Sigma_g^b$ and a homeomorphic copy ${\Sigma_g^b}'$ identified (possibly up to an orientation reversing homeomorphism) along their associated boundary curves (cf. Figure \ref{double}). Thus, $D(\Sigma_g^b)$ is a closed surface. Let $\delta=\delta_1\cup ...\cup \delta_b$ denote the simple closed curves on $D(\Sigma_g^b)$ that correspond to the identified boundary components. Then any finite Abelian cover $p\colon \widehat\Sigma\to D(\Sigma_g^b)$ is a closed surface and the arguments used in Lemma \ref{scott} that allow for decreasing the number of self-intersections upon elevating $x\subset D(\Sigma_g^b)$ are carried out analogously. We denote by $R_1,...,R_m$ the set of subsurfaces of $\widehat \Sigma$ bounded by $p^{-1}(\delta)=\widehat\delta_1\cup ...\cup\widehat\delta_l$. Then there exists some $i\in \{1,...,m\}$ such that $p|_{R_i}\colon R_i\to \Sigma_g^b$ is a finite cover with boundary to which $x$ elevates as a simple closed curve. This shows the corollary.	
\end{proof}

\begin{figure}[H]
\centering
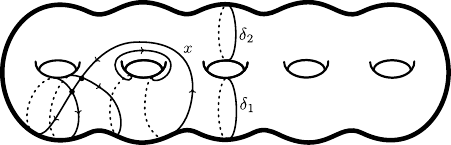
\caption{The double $D(\Sigma_2^2)$ of the genus $2$ surface with $2$ boundary components.}\label{double}
\end{figure}

	\begin{mydef}\label{submodule}
	Let $\Sigma$ now be a compact orientable surface of positive genus (possibly with non-empty boundary)   and let $\Sigma'\to \Sigma$ be a finite regular cover. In the punctured case, let $\Sigma'\hookrightarrow \overline{\Sigma'}$ the puncture filling map. We denote by $V_{x}^{\Sigma'}$ the submodule of the symplectic $\mathbb{Z}$-module $H_1(\overline{\Sigma'},\mathbb{Z})$ that is generated by the homology classes of elevations of a nonperipheral closed curve $x\subset \Sigma$ to $\Sigma'$, i.e. $$V_{x}^{\Sigma'}:=\mathrm{Span}_{\IZ}\{g[x']\mid g\in \Gamma, x'\ \text{preferred elevation}\}\subseteq H_1(\overline{\Sigma'},\mathbb{Z}).$$
	\end{mydef}
	Since $\Gamma$ acts canonically and transitively on the set of lifts of elevations of $x$, it follows that  $V_{x}^{\Sigma'}$ is a $\Gamma$-invariant submodule of $H_1(\overline{\Sigma'},\mathbb{Z})$. Going forward, we denote by $\mathcal S$ the class of finite regular covers of $\Sigma$. The following result is then the analog of \cite[Theorem 2.3]{Boggi} and highlights that covering towers are well adapted to detect whether two closed curves on a topological surface are disjoint.

\begin{lemma}\label{boggi}
Let $x,y$ be two (not necessarily distinct) closed curves on $\Sigma$. Then $x$ and $y$ are disjoint if and only if $\widehat\iota([x'],[y'])=0$ for all $[x'] \in V_{x}^{\Sigma'}$ and $[y'] \in V_{y}^{\Sigma'}$ for every $\Sigma'\in \mathcal S$.
\end{lemma}

\begin{proof}
	The first direction is clear since if $x$ and $y$ are disjoint, then this is true for all elevations of $x$ and $y$. For the converse direction, let $x,y\subset \Sigma$ be two closed curves that are not disjoint and intersect in minimal position. By Lemma \ref{scott} there exists an iterated finite Abelian cover $\Sigma'\in \mathcal S$ such that all elevations of $x$ and $y$ to $\Sigma'$ are simple closed curves. Let $x'$ and $y'$ be two simple elevations of $x$ and $y$ respectively that are not disjoint on $\Sigma'$. Such elevations exist since the intersection of $x$ and $y$ in $\Sigma$ is witnessed by their elevations. Now if they have algebraic intersection number $\widehat\iota([x'],[y'])\not= 0$, we are done. If not, then we find some simple subarc $a \subset x'$ that intersects $y'$ only in its endpoints and returns to the same side of $y'$ (cf. Figure \ref{concat}). Let $b\subset y'$ be a subarc with the same endpoints as $a$ and with matching orientation at the common endpoints (compare the argument in Lemma \ref{scott}, Case 2B), then $a\cup b$ defines an oriented simple closed curve on $\Sigma'$. Following the arguments in the proof of Lemma \ref{scott} we then find an Abelian cover of $\Sigma'$ and elevations of $x'$ and $y'$ so that the geometric intersection number is strictly smaller than that of $x'$ and $y'$. Inductively we thus find an iterated cover $\widehat\Sigma\to \Sigma$ in which two simple elevations $\widehat x$ and $\widehat y$ of $x$ and $y$ respectively have geometric intersection $i(\widehat x,\widehat y)=1$ and thus algebraic intersection $\widehat\iota([\widehat x],[\widehat y])\not=0$. This finishes the proof. \end{proof}
	
		\begin{figure}[H]
\centering
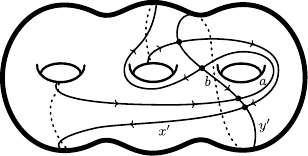
\caption{The two subarcs $a\subset x'$ and $b\subset y'$ form an oriented simple closed curve $a\cup b$ on $\Sigma'$.}\label{concat} 
\end{figure}
	
By taking $x=y$ in Lemma \ref{boggi} it follows that covering towers furthermore detect simplicity of closed curves which is the analog of \cite[Corollary 2.4]{Boggi}.

\begin{corollary}\label{boggi2}
	Let $x\subset \Sigma$ be a closed curve.  Then $x$ is freely homotopic to the power of a simple closed curve if and only if the symplectic submodule $V_{x}^{\Sigma'}$ is totally isotropic for every $\Sigma'\in \mathcal S$.
\end{corollary}

\section{Detecting Surface Homeomorphisms in Free Homotopy Classes}\label{final_section}

In this final section we give a proof of Theorem \ref{mcg}. For the remainder of this section we consider a punctured surface $\Sigma_g^b$ that has the same homotopy type as $R_n$, so that we have an identification of $\Out(\pi_1(\Sigma_g^b))$ with $\Out(F_n)$. Now let $\mathcal S$ denote the class of finite characteristic covers of $\Sigma_g^b$. The choice of the surface $\Sigma_g^b$ determines a choice of a symplectic form on $H_1(\Sigma_g,\mathbb{Z})$ via the algebraic intersection pairing $\widehat\iota(\cdot,\cdot)$ which therefore endows the modules $H_1(\overline{\Sigma'},\mathbb{Z})$ with a symplectic structure for every cover $\Sigma'\to \Sigma_g^b$ in $\mathcal S$ (compare Definition \ref{submodule}). \\

Theorem \ref{mcg} gives a homological criterion to detect whether the homotopy equivalence of $R_n$ that represents $\psi\in \Out(F_n)$ is contained in the free homotopy class of a homeomorphism of $\Sigma_g^b$ via the virtual homology representations of $\Out(F_n)$: $$\rho_{\mathcal S}\colon \Out(F_n)\to \prod_{\Sigma'\in \mathcal S}\mathrm{GL}(H_1(\Sigma',\IZ))/\Gamma_{\Sigma'}.$$

Note that the action of $\Out(F_n)$ does not fix a basepoint in $\Sigma_g^b$ and lifts of homotopy equivalences are therefore only well defined up to the action of the deck group $\Gamma_{\Sigma'}$ for each $\Sigma'\in \mathcal S$. 

\begin{proof}[Proof of Theorem \ref{mcg}]
Let $f\colon R_n\to R_n$ be the homotopy equivalence that represents $\psi \in \Out(F_n)$ and suppose that $\psi$ acts on $H_1(\Sigma',\IZ)$ by preserving the algebraic intersection pairing for every cover $\Sigma'\in \mathcal S$ as stated in the theorem. As a consequence of Lemma \ref{boggi}, it follows that the trivial geometric intersection between any two disjoint curves on $\Sigma_g^b$ is preserved by $\psi$. Additionally, for any simple closed curve $x\subset \Sigma_g^b$, the symplectic submodule $V_x^{\Sigma'}$ of $H_1(\overline{\Sigma'},\IZ)$ is preserved as totally isotropic for every $\Sigma'\in \mathcal S$. By Corollary \ref{boggi2} it thus follows that $\psi$ preserves the conjugacy classes of peripheral simple closed curves on $\Sigma_g^b$, since peripheral elements are characterized by the property that they have trivial geometric intersection with any other closed curve. By the Theorem of Dehn-Nielsen-Baer for punctured surfaces (see \cite[Theorem 8.8]{Primer}) it now follows that $\psi$ is realized by a homeomorphism of $\Sigma_g^b$ that is homotopic to $f$. This proves the claim.
\end{proof}

\bibliographystyle{amsalpha}
\bibliography{references}

\end{document}